\documentclass[12pt]{amsart}
\usepackage{amsmath,amsthm,amssymb}
\usepackage{mathptmx,mathrsfs,enumerate,url}
\usepackage{comment}
\usepackage{datetime}

\usepackage[backend=biber,style=numeric,sorting=nyvt]{biblatex}
\usepackage[all]{xy}

\newcommand{\dd}[1]{\[\xymatrix{#1}\]}

\usepackage{tikz-cd}

\usepackage{extpfeil}
\usepackage{mathtools}

\newtheorem{theorem}{Theorem}[section]
\newtheorem{prop}[theorem]{Proposition}
\newtheorem{corollary}[theorem]{Corollary}
\newtheorem{lem}[theorem]{Lemma}
\newtheorem{introtheorem}{Theorem}

\theoremstyle{definition}
\newtheorem{definition}[theorem]{Definition}
\newtheorem{remark}[theorem]{Remark}

\newenvironment{acknowledgements}%
	{
   \begin{center}%
	\scshape Acknowledgments
   \end{center}\vspace{6pt}
   }%

\def\bZ{{\mathbb Z}}
\def\bQ{{\mathbb Q}}

\def\Qp{{{\mathbb Q}_p}}

\def\CalO{{\mathcal O}}

\def\End{{\mbox{End}}}
\def\Aut{{\mbox{Aut}}}

\def\Out{{\mbox{Out}}}
\def\Gal{{\mbox{Gal}}}

\def\Nm{{\mathrm{Nm}}}

\def\defeq{ \ {\stackrel{\mathrm{def}}{=}} \ }
\def\nswaut{{{}^*\alpha}}

\title[On Outer Automorphism Groups]{On the Outer Automorphism Groups of the Absolute Galois Groups of 2-adic local Fields}

\author[Yu Nishio]{Yu Nishio \\ \\ December 2025}
\address[Yu Nishio]{}
\email{nishio.math@gmail.com}



\begin{document}

\begin{abstract}
In the present paper, we study the outer automorphism groups of the absolute Galois groups of 2-adic local fields from the point of view of anabelian geometry. Let us recall that it is well-known that the natural homomorphism from the automorphism group of a mixed-characteristic local field to the outer automorphism group of the absolute Galois group of the given mixed-characteristic local field is injective. 
Moreover, Hoshi and the author of the present paper proved that, for absolutely abelian mixed-characteristic local fields with odd residue characteristic $p$ and even extension degree over $\bQ_p$, this subgroup arising from field automorphisms is not normal in the outer automorphism group and has infinitely many distinct conjugates. 
As a result in this direction, one of the main results of the present paper is the assertion that if a 2-adic local field satisfies certain conditions, then the set of conjugates of the subgroup arising from field automorhisms in the outer automorphism group is infinite, which thus implies that this subgroup is not normal in the outer automorphism group. 
On the other hand, for an odd prime number $p$, Hoshi proved the existence of an irreducible Hodge-Tate $p$-adic representation of dimension two of the absolute Galois group of a $p$-adic local field and an automorphism of the absolute Galois group such that the $p$-adic Galois representation obtained by pulling back the given $p$-adic Galois representation by the given automorphism is not Hodge-Tate. 
As a result in a direction that is different from the above direction, another main result of the present paper is the existence of a pair of a representation and an automorphism that is
 similar to the above pair in the case where $p=2$.  
\end{abstract}

\maketitle

\setcounter{section}{-1}

\section*{Introduction}
Let $p$ be a prime number, $k$ a finite extension of 
$\mathbb{Q}_p$, and $\overline{k}$ an algebraic closure 
of $k$. Write $G_k \stackrel{\mathrm{def}}{=} \mathrm{Gal}
(\overline{k} / k)$ for the absolute Galois group of $k$ 
determined by the algebraic closure $\overline{k}$ 
and $\mathrm{Out}(G_k)$ for the group of outer 
automorphisms of the group $G_k$. 
In the present paper, we study the outer automorphism group $\mathrm{Out}(G_k)$ from the point of view of anabelian geometry. Write $\mathrm{Aut}(k)$ for the group of automorphisms of the field $k$. Thus, we have a natural homomorphism 
$\mathrm{Aut}(k) \to \mathrm{Out}(G_k)$ of groups. 
Let us first recall that it is well-known [cf., e.g., \cite{Hoshi1}, 
Proposition 2.1] that this homomorphism is injective. In 
the present paper, let us regard $\mathrm{Aut}(k)$ as a 
[necessarily finite] subgroup of $\mathrm{Out}(G_k)$ by 
means of this injective homomorphism: 
\dd{
\mathrm{Aut}(k) \subseteq \mathrm{Out}(G_k). 
} 
Here, we note that it is well-known [cf., e.g., the 
discussion given at the final portion of \cite{NSW}, Chapter VII, 
\S 5; \cite{Hoshi-Nishio}, Corollary 1.6, (iv)] that, if $p$ is odd, then the equality $\mathrm{Aut}(k) = \mathrm{Out}(G_k)$ never holds.
In particular, one may conclude that 
a finite extension of $\bQ_p$ 
should be considered to be ``not anabelian'' [cf.\ also \cite{NSW}, Chapter XII, \S 2, Closing remark]. 
Therefore, one main interest from the point of view of anabelian
geometry is in the investigation of the ``difference'' between 
$\mathrm{Aut}(k)$ and $\mathrm{Out}(G_k)$. 
Some results concerning the ``characterization'' of the subgroup 
$\mathrm{Aut}(k)$ of $\mathrm{Out}(G_k)$ may be found in 
\cite{Mzk}, \S 3, and \cite{Hoshi-b}, \S 3. Moreover, some results concerning this ``difference'' may be found in 
\cite{Hoshi2}, 
\cite{Hoshi-Nishio}, and \cite{kondo} for odd prime numbers $p$. 
On the other hand, no similar results were known in the case where $p=2$. 
Thus, it is a natural question to ask how far $\Aut(k)$ differs from $\Out(G_k)$ in the case where $p=2$. The present paper gives a certain answer to this question.
Write $(\mathbb{Q}_p)_+ \subseteq k_+$ for the underlying 
additive modules of the fields $\mathbb{Q}_p \subseteq k$, 
respectively. Next, let us recall that, by applying a 
functorial group-theoretic reconstruction algorithm 
established in the study of the mono-anabelian geometry of mixed-characteristic local fields [cf., 
e.g., \cite{Hoshi1}, Definition 3.10, (vi), and \cite{Hoshi1}, 
Proposition 3.11, (iv)], one obtains an action of the group 
$\mathrm{Out}(G_k)$ on the module $k_+$ whose restriction 
to the subgroup $\mathrm{Aut}(k) \subseteq 
\mathrm{Out}(G_k)$ coincides with the natural action of 
$\mathrm{Aut}(k)$ on $k_+$.

One main technical result of the present paper is as 
follows [cf.\ Theorem \ref{one:theorem:2.8}]:

\begin{introtheorem}
\label{AAAAAAA}
Suppose that the following two conditions are 
satisfied: 
\begin{itemize}
\item[$(1)$]
The field $k$ contains a primitive $4$-th root of unity. 
\item[$(2)$]
The finite extension $k / \mathbb{Q}_2$ is Galois, 
and, moreover, the Galois group 
$\mathrm{Gal}(k / \mathbb{Q}_2)$ is abelian. 
\end{itemize}
Then there exists an outer automorphism $\alpha$ of $G_k$ 
such that, for each nonzero integer $n$, if one writes 
$\alpha_+^n$ for the action of $\alpha^n$ on $k_+$, then 
$\alpha^n_+((\mathbb{Q}_2)_+) \neq (\mathbb{Q}_2)_+$. 
\end{introtheorem}

Moreover, in the present paper, we prove the 
following result [cf.\ Theorem \ref{two:theorem:2.8}]:

\begin{introtheorem}
\label{BBBBBBB}
Suppose that the two conditions in the statement of 
Theorem~\ref{AAAAAAA} are satisfied. Then the set of 
$\mathrm{Out}(G_k)$-conjugates of the subgroup 
$\mathrm{Aut}(k) \subseteq \mathrm{Out}(G_k)$ is 
infinite. 
\end{introtheorem}

A formal consequence of Theorem~\ref{BBBBBBB} is as 
follows [cf.\ Corollary \ref{corollary:2.9}]: 

\begin{introtheorem}
\label{CCCCCCC}
Suppose that the two conditions in the statement of 
Theorem~\ref{AAAAAAA} are satisfied. 
Then the following hold: 
\begin{itemize}
\item[$(\mathrm{i})$]
The subgroup $\mathrm{Aut}(k) \subseteq 
\mathrm{Out}(G_k)$ is not normal. 
\item[$(\mathrm{ii})$]
There exist infinitely many distinct 
subgroups of $\mathrm{Out}(G_k)$ isomorphic 
to $\mathrm{Aut}(k)$. 
\end{itemize}
\end{introtheorem}

The final main result of the present paper is as follows [cf. Corollary \ref{cor:4.4}]:

\begin{introtheorem}
\label{DDDDDDD}
There exist a finite extension $K$ of $\bQ_2$, an algebraic closure $ \overline{K}$ of $K$, a $\mathbb{Q}_2$-vector space $V$ of dimension 2, and a continuous representation $\rho : \mathrm{Gal}(\overline{K}/K) \to \mathrm{Aut}_{\mathbb{Q}_{p_K}}(V)$ that is irreducible, abelian, crystalline [hence also Hodge-Tate], but not Aut-intrinsically Hodge-Tate [cf. \cite{Hoshi_intrinsic_Hodge-Tate}, Definition 1.3]. 
\end{introtheorem}

In \cite{kondo}, Kondo proved similar results to Theorem \ref{AAAAAAA}, Theorem \ref{BBBBBBB}, Theorem \ref{CCCCCCC}, and Theorem \ref{DDDDDDD} in the case where $k$ is absolutely Galois [i.e., where the extension $k/\bQ_p$ is Galois], and $p$ is odd. 
Moreover, in \cite{kumpitsch}, Kumpitsch studied the outer automorphism groups of the absolute Galois groups of mixed-characteristic local fields from the point of view of the study of mapping class groups and constructed large subgroups of the outer automorphism groups.

\section{Notational conventions}
\label{section0}

\subsubsection*{\sc Topological groups} 

If $G$ is a topological group, then we shall write $G^{\rm ab}$ for the \textit{abelianization} of $G$ [i.e., the quotient of $G$ by the closure of the commutator subgroup
of $G$], $G^{\textrm{ab-tor}}$ for the closure of the subgroup of torsion elements of $G^{\rm ab}$, and $G^{\rm ab/tor}$ for the quotient of $G^{\rm ab}$ by $G^{\textrm{ab-tor}}$. 
If $H$ is a profinite group, and $p$ is a prime number, then we shall write
$H^{(p)}$ for the \textit{maximal pro-$p$ quotient} of $H$.

\subsubsection*{\sc Rings}

If $R$ is a ring, then we shall write $R_+$ for the underlying additive module of $R$ and $R^\times$ for the multiplicative group of units of $R$.  

\subsubsection*{\sc Modules}

Let $M$ 
be a module. We shall write 
\[
M^{\wedge} \stackrel{\mathrm{def}}{=} \varprojlim\, M/(n\cdot M)
\]
--- where the projective limit is taken over the positive integers $n$.

\subsubsection*{\sc Fields} 

We shall refer to a field isomorphic to a finite extension of $\Qp$, for some prime number $p$, as a {\it mixed-characteristic local field}.  

If $k$ is a mixed-characteristic local field, then we shall write 
\begin{itemize}
 \item  $k^{(d=1)} \subseteq k$ for the [uniquely determined] minimal mixed-characteristic local field contained in $k$,
 \item $f_k$ for the absolute residue degree of $k$, 
 \item $d_k \defeq [k:k^{(d=1)}]$ for the degree of the finite extension $k/k^{(d=1)}$,  
 \item $\CalO_k \subseteq k$ for the ring of integers of $k$, 
 \item $\mathfrak{m}_k \subseteq \CalO_k$ for the maximal ideal of $\CalO_k$, 
 \item $\CalO^{\prec}_k \defeq 1+\mathfrak{m}_k \subseteq \CalO^{\times}_k$, 
 \item $p_k$ for the residue characteristic of $k$, and 
 \item $a_k$ for the largest nonnegative integer such that $k$ contains a primitive $p^{a_k}_k$-th root of unity.
\end{itemize}

We shall refer to a group isomorphic to the absolute Galois group of 
a mixed-characteristic local field as a group {\it of MLF-type} 
[cf.\ \cite{Hoshi4}, Definition 1.1].    
In the present paper, let us always regard a group of MLF-type as a 
profinite group by means of the profinite topology 
discussed in \cite{Hoshi4}, Proposition 1.2, (i)
[cf.\ also \cite{Hoshi4}, Proposition 1.2, (ii)].

\section{Existence of an automorphism with a certain unipotency condition of a group of  MLF-type corresponding to a 2-adic local field containing primitive 4-th roots of unity}
\label{section1}

In the present \S \ref{section1}, we prove that a certain group of MLF-type admits an automorphism that
satisfies a certain unipotency condition [cf.\ Theorem~\ref{theorem:1.5}
below].
In the present \S \ref{section1}, let $G$ be a group of MLF-type.  
Thus, by applying the various functorial group-theoretic reconstruction algorithms of \cite{Hoshi1}, \S 3 [cf.\ \cite{Hoshi1}, Definition 3.5, (i), (ii), (iii); \cite{Hoshi1}, Definition 3.10, (ii), (iv), (vi)], to  the group $G$ of MLF-type, we obtain
\begin{itemize}
 \item a prime number $p(G)$,
 \item positive integers $d(G)$ and $f(G)$, 
  \item a nonnegative integer $a(G) \defeq \log_{p(G)} \left( \sharp \left( \left(k^{\times}(G)_{\mathrm{tor}}\right)^{(p(G))} \right) \right)$
 \item a normal closed subgroup $P(G) \subseteq 
 G$ of $G$, and
 \item topological modules $\CalO^{\prec}(G) \subseteq k^\times (G)$ and $k_+(G)$ 
\end{itemize}
[cf.\ also \cite{Hoshi1}, Summary 3.15]. 
Here, let us recall [cf.\ \cite{Hoshi1}, Proposition 3.6; 
\cite{Hoshi1}, Proposition 3.11, (i), (iv); \cite{Hoshi2}, Proposition 2.5] that if $k$ is a 
mixed-characteristic local field, 
$\overline{k}$ is an algebraic closure of $k$, 
and $G_k \defeq \mathrm{Gal}(\overline{k}/k)$ 
is the absolute Galois group of $k$ 
determined by the algebraic closure $\overline{k}$, then 
\begin{itemize}
 \item $p(G_k)$, $d(G_k)$, $a(G_k)$, $f(G_k)$, $P(G_k)$ coincide with
 $p_k$, $d_k$, $a_k$, $f_k$, the wild inertia subgroup of $G_k$, 
 respectively, and, moreover, 
 \item there exist functorial isomorphisms $\CalO^{\prec}_k \stackrel{\sim}{\to} \CalO^{\prec}(G_k)
 $, $k^\times\stackrel{\sim}{\to} k^\times (G_k) $, and $k_+
 \stackrel{\sim}{\to} k_+(G_k)$.  
\end{itemize}
Furthermore, let us also recall [cf.\ 
also 
\cite{Hoshi1}, Proposition 3.6; 
\cite{Hoshi1}, Definition 3.10, (i), (ii), (vi);
\cite{Hoshi1}, Proposition 3.11, (i)]
that 
\begin{itemize}
 \item 
the module $\CalO^{\prec}(G)$ is defined 
to be the image of $P(G) \subseteq G$ in $G^{\rm ab}$, and 
 \item
the module $k_+(G)$ is defined to be 
$\CalO^{\prec}(G) \otimes_{\mathbb{Z}} \mathbb{Q}$, i.e., 
the perfection [cf., e.g., the discussion entitled ``Monoids'' of 
\cite{Hoshi2}, \S 0] of the module $\CalO^{\prec}(G)$ [cf.\ also 
the commutative diagram 
\dd{
\mathcal{O}_k^{\prec}
\ar[r]^{\log_k} 
\ar[d]_\wr &
k_+
\ar[d]_\wr 
\\
\mathcal{O}^{\prec}(G_k)
\ar[r]&
k_+(G_k)
}
determined by the commutative diagram of 
\cite{Hoshi1}, Proposition 3.11, (iv)].  
\end{itemize}

In the present \S \ref{section1}, suppose, moreover, that $p(G) = 2$, and that $a(G) \geq 2$. 

\begin{prop}[Diekert]\label{proposition:1.1} 
There exist 
\begin{itemize}
 \item
topological generators 
$\sigma$, $\tau$, $x_0, \dots, x_{d(G)}$ of $G$, 
 \item
positive integers $s$, $g$ such that $g \neq 1$, and 
\end{itemize}
that satisfy the following three conditions:  
\begin{enumerate}
 \item[\rm (1)]  
 The normal closed subgroup $P(G)$ of
 $G$ is topologically normally generated by $x_0, \dots, x_{d(G)}$.
 \item[\rm (2)] The equality $\sigma \tau \sigma^{-1} = \tau^{q(G)}$ holds, where we write $q(G) \defeq 2^{f(G)}$.
 \item[\rm (3)] The equality 
\dd{
\sigma x_0 \sigma^{-1}
=
(x_0 \tau)^{\pi g}
x_1^{2 ^{s}}
\left[x_1,x_2\right] \cdots \left[x_{d(G)-1},x_{d(G)}\right]
}
holds, where we write $\pi$ for the unique element of $\hat{\bZ} = \prod_p \bZ_p$ whose image in $\bZ_p$ is given by $1$ if $p = 2$ $($resp.\ by $0$ if $p\neq 2 $$)$.  
\end{enumerate}
\end{prop}
\begin{proof}
This assertion follows from \cite{Diekert1984}, Theorem 3.1, and \cite{Hoshi1}, Proposition 3.6, \cite{Hoshi2}, Proposition 2.5, (i). 
\end{proof}

In the remainder of the present \S \ref{section1}, let us fix 
topological generators $\sigma$, $\tau$, $x_0, \dots, x_{d(G)}$ of 
$G$ as in Proposition \ref{proposition:1.1}. Write $S \defeq \{0, 
1, \dots, d(G)\}$. 
Moreover, for each $i \in S$, write 
\begin{itemize}
\item
$y_i \in k_+(G)$ for the image of $x_i \in P(G)$ 
[cf.\ the condition (1) of Proposition \ref{proposition:1.1}]
in $k_+(G)$ 
and 
\item
$z_i \in \CalO^{\prec}(G)^{\rm ab/tor}$ for the image 
of $x_i  \in P(G)$ 
[cf.\ the condition (1) of Proposition \ref{proposition:1.1}]
in $\CalO^{\prec}(G)^{\rm ab/tor}$ 
\end{itemize}
[cf.\ the constructions of $\CalO^{\prec}(G)$, $k_+(G)$ 
explained in the discussion 
preceding Proposition \ref{proposition:1.1}]. 

\begin{lem}\label{lemma:1.2} 
The topological module $k_{+}(G)$ has a natural structure of ${\bQ}_{2}$-vector space of dimension $d(G)$. 
\end{lem}

\begin{proof}
It follows immediately from the definition of $k_{+}(G)$ [cf.\ \cite{Hoshi1}, Definition 3.10, (vi)] that $k_{+}(G)$ has a natural structure of ${\bQ}_{2}$-vector space. Moreover, it follows from \cite{Hoshi1}, Proposition 3.6, and  \cite{Hoshi1}, Proposition 3.11, (iv), that this $\bQ_{2}$-vector space $k_{+}(G)$ is of dimension $d(G)$. 
\end{proof}

\begin{lem}\label{lemma:1.3} 
The $d(G)$ elements $y_1, \dots, y_{d(G)}$ form a basis of the $\bQ_{2}$-vector space $k_{+}(G)$ [cf.\ Lemma \ref{lemma:1.2}]. 
\end{lem}

\begin{proof}
Since ${\{x_i\} }_{i \in S}$ topologically normally generates $P(G)$ [cf.\ the condition (1) of Proposition \ref{proposition:1.1}], ${\{z_i\} }_{i \in S}$ topologically generates $\CalO^{\prec}(G)^{\rm ab/tor}$ $( \subseteq G^{\rm ab/tor} )$ [cf.\ \cite{Hoshi1}, Definition 3.10, (i), (ii)]. Moreover, since $\CalO^{\prec}(G)^{\rm ab/tor} \otimes_{\mathbb{Z}_{2}} (\mathbb{Z}_{2} / 2^n\mathbb{Z}_{2})$ is a finite $2$-group for every positive integer $n$ [cf.\ \cite{Hoshi1}, Proposition 3.11, (i)], ${\{z_i\} }_{i \in S}$ is also a generator of $\CalO^{\prec}(G)^{\rm ab/tor}$ even if we regard $\CalO^{\prec}(G)^{\rm ab/tor}$ as a $\bZ_{2}$-module. Next, let us observe that it follows from the condition (2) of Proposition \ref{proposition:1.1} that the image of $\tau$ in $G^{\rm ab/tor}$ is trivial. Thus, it follows from the condition (3) of Proposition \ref{proposition:1.1} that the relation $1={z_0}^{H}{z_1}^{{2}^s}$ in $\CalO^{\prec}(G)^{\rm ab/tor}$ holds for some nonzero [cf.\ the condition that $g \neq 1$ of Proposition \ref{proposition:1.1}] integer $H$. Therefore, if we write $T \defeq S \backslash \{0\}$, then ${\{z_i \otimes 1 \} }_{i \in T}$ is a generator of $\bQ_{2}$-vector space $\CalO^{\prec}(G)^{\rm ab/tor} \otimes_{\bZ_{2}}  \bQ_{2}$. Here, let us observe that we have a natural topological isomorphism $k_{+}(G) \stackrel{\sim}{\to} \CalO^{\prec}(G)^{\rm ab/tor} \otimes_{\bZ_{2}}  \bQ_{2}$, by definition, that maps $y_i \in k_{+}(G)$ to $z_i \otimes 1 \in  \CalO^{\prec}(G)^{\rm ab/tor} \otimes_{\bZ_{2}} \bQ_{2}$ for each $i \in S$. This isomorphism is also an isomorphism of $\bQ_{2}$-vector spaces by construction. Moreover, since $k_{+}(G) \simeq \CalO^{\prec}(G)^{\rm ab/tor} \otimes_{\bZ_{2}}  \bQ_{2}$ is a $\bQ_{2}$-vector space of dimension $d(G)$ [cf.\ Lemma \ref{lemma:1.2}], ${\{{z_i} \otimes 1 \} }_{i \in T}$ is a basis of the $\bQ_{2}$-vector space $\CalO^{\prec}(G)^{\rm ab/tor} \otimes_{\bZ_{2}} \bQ_{2}$. Therefore, it follows from the condition imposed on the isomorphism $k_{+}(G) \stackrel{\sim}{\to} \CalO^{\prec}(G)^{\rm ab/tor} \otimes_{\bZ_{2}} \bQ_{2}$ that ${\{y_i\} }_{i \in T}$ is a basis of the $\bQ_{2}$-vector space $k_{+}(G)$. 
\end{proof}

\begin{lem}\label{lemma:1.4} 
The element $y_{d(G)-1}$ is a $nonzero$ element of $k_{+}(G)$. 
\end{lem}
\begin{proof} 
This assertion follows from Lemma \ref{lemma:1.3} [cf.\ also our assumption that $a(G) > 1$, hence $d(G) > 1$]. 
\end{proof}

\begin{definition} \label{alpha}
We shall write $\nswaut$ for the automorphism of $G$ 
defined by the equalities $\nswaut(\sigma) = \sigma$, $\nswaut(\tau) = \tau$, $\nswaut(x_{d(G)}) = x_{d(G)} x_{d(G)-1}$, and $\nswaut(x_i) = x_i$ for $i\in S \backslash \{d(G)\}$ . Since $\nswaut$ induces a bijection between the set of generators of $G$ and preserves the defining relations of $G$ [cf. Proposition \ref{proposition:1.1}], it indeed defines an automorphism of $G$.
\end{definition}

\begin{theorem}\label{theorem:1.5}
Let $G$ be a group of MLF-type such that $p(G) = 2$, and that $a(G) \geq 2$.  Then, for each nonzero integer $n$, if one writes $\nswaut^n_+$ for the automorphism of the $\mathbb{Q}_{2}$-vector space $k_+(G)$ induced by $\nswaut^n$ [cf. Definition \ref{alpha}], then $\nswaut^n_+ \neq \mathrm{id}$, and, moreover, the equality $(\nswaut^n_{+}- {\rm id})^2 = 0$ in the ring of endomorphisms of $k_+(G)$ holds.  
\begin{proof} 
First, we prove that $\nswaut_+^n \neq \mathrm{id}$ for each nonzero integer $n$. If $\nswaut_+^n = {\rm id}$, then $y_{d(G)} = \nswaut^n_+(y_{d(G)}) = y_{d(G)} + ny_{d(G)-1}$, which thus implies that $y_{d(G)-1} = 0$ in $k_{+}(G)$. However, this contradicts Lemma \ref{lemma:1.4}. Thus, we conclude that $\nswaut_+^n \neq {\rm id}$. Next, let us observe that, for each nonzero integer $n$, it follows from the easily verified equality $(\nswaut^n_+ - {\rm id})^2 (y_i) = 0$ for every $i \in S$ and Lemma \ref{lemma:1.3} that the equality $(\nswaut^n_+- {\rm id})^2 = 0$ in $\End(k_+(G))$ holds, as desired. This completes the proof of Theorem \ref{theorem:1.5}. 
\end{proof}
\end{theorem}

\begin{corollary}\label{corollary:2.12}
Let $G$ be a group of MLF-type such that $p(G) = 2$, and $a(G) \geq 2$. 
Then the following hold: 
\begin{enumerate}
 \item[\rm (i)] The image of the natural homomorphism from the outer automorphism 
group of $G$ to the automorphism group of $k_+(G)$ is infinite. \label{corollary:2.12:statement:1}
 \item[\rm (ii)] The image of the natural homomorphism from the outer automorphism 
group of $G$ to the automorphism group of $G^{\mathrm{ab}}$ is infinite. \label{corollary:2.12:statement:2}
 \item[\rm (iii)] The image of the natural homomorphism from the outer automorphism 
group of $G$ to the automorphism group of $k^\times(G)$ is infinite. \label{corollary:2.12:statement:3}
 \item[\rm (iv)] The outer automorphism group of $G$ is infinite. \label{corollary:2.12:statement:4}
\end{enumerate}
\end{corollary}
\begin{proof}
Assertion (i) follows from Theorem~\ref{theorem:1.5}.  
Assertion (ii) follows from assertion (i), 
together with the definition of $k_+(G)$ 
[cf.\ \cite{Hoshi1}, Definition 3.10, (vi)].  
Assertion (iii) follows from assertion (ii) and 
the [easily verified] density of $k^\times(G)$ in $G^{\mathrm{ab}}$ 
[cf.\ \cite{Hoshi1}, Definition 3.10, (iv)].  
Assertion (iv) follows from assertion (i).  
\end{proof}

\begin{remark}\label{remark:2.13}
Let us recall that it follows immediately from 
\cite{Hoshi2}, Corollary 5.5, that each of the three images 
discussed in Corollary \ref{corollary:2.12}, 
(i), (ii), (iii), in the case where $d(G)$ is
equal to $1$ is trivial.
\end{remark}

\section{Existence of a special automorphism of the absolute Galois group of an absolutely abelian 2-adic local field containing primitive $4$-th roots of unity}
\label{section2}

In the present \S\ref{section2}, we prove that the absolute Galois group of a certain 2-adic local field admits
an automorphism that has an interesting property [cf.\ Theorem \ref{one:theorem:2.8} below]. In  the present \S \ref{section2}, let $k$ be a mixed-characteristic local field and $\overline{k}$ an algebraic closure of $k$. 
We shall write
$G_k \defeq \mathrm{Gal}(\overline{k}/k)$ for the absolute Galois group of $k$ determined by the algebraic closure $\overline{k}$. 
Write, moreover, $\mathrm{Aut}(G_k)$, $\mathrm{Aut}(k_+)$, and $\mathrm{Aut}(k^\times)$ for the groups of automorphisms of the group $G_k$, the module $k_+$, and the module $k^\times$, respectively.  Thus, it follows from \cite{Hoshi1}, Proposition 3.11, (i), (iv), that we have natural homomorphisms 
\dd{
\Aut(G_k) \ar[r] & \Aut(k_+), & \Aut(G_k) \ar[r] & \Aut(k^\times) .  
}

\begin{definition}\label{definition:2.1}    
      \ \ \ 
\begin{enumerate}
  \item[\rm (i)]  We shall say that $\alpha \in \Aut(k_+)$ is $(\bQ_{p_k})_+$-{\it characteristic} if $\alpha(k^{(d=1)}_+) = k^{(d=1)}_+$. \label{definition:2.1:statement1} 
  \item[\rm (ii)] We shall say that $\alpha \in \Aut(k_+)$ is $(\bQ_{p_k})_+$-{\it preserving} if $\alpha$ is $(\bQ_{p_k})_+$-characteristic, and $\left. \alpha \right|_{k^{(d=1)}_+}$ is the identity automorphism of $k_+^{(d=1)}$.  \label{definition:2.1:statement:2} 
  \item[\rm (iii)]  We shall say that $\alpha \in \Aut(k_+)$ is {\it group-theoretic} if $\alpha$ is contained in the image of the first homomorphism $\Aut(G_k) \to \Aut(k_+)$ of the above display. \label{definition:2.1:statement:3} 
\item[\rm (iv)]  We shall say that $\alpha \in \Aut(k^{\times})$ is {\it group-theoretic} if $\alpha$ is contained in the image of the second homomorphism $\Aut(G_k) \to \Aut(k^\times)$ of the above display. \label{definition:2.1:statement:4}    
  \item[\rm (v)]  We shall say that $\alpha \in \Aut(G_k)$ is $(\bQ_{p_k})_+$-{\it characteristic} if the group-theoretic automorphism of $k_+$ induced by $\alpha$ is $(\bQ_{p_k})_+$-characteristic. \label{definition:2.1:statement:5} 
  \item[\rm (vi)] We shall say that $\alpha \in \Aut(G_k)$ is $(\bQ_{p_k})_+$-{\it preserving} if the group-theoretic automorphism of $k_+$ induced by $\alpha$ is $(\bQ_{p_k})_+$-preserving.  \label{definition:2.1:statement:6} 
\end{enumerate}
\end{definition}

\begin{lem}\label{lemma:2.2}
Let $\alpha$ be an automorphism of $G_k$.  Write $\alpha_+ \in \mathrm{Aut}(k_+)$ and $\alpha^{\times} \in \mathrm{Aut}(k^{\times})$ for the respective group-theoretic automorphisms induced by $\alpha$, 
${\rm Nm}_{k/k^{(d=1)}}$
for the norm map with respect to the finite extension $k/k^{(d=1)}$, 
and ${\rm Tr}_{k/k^{(d=1)}}$
for the trace map with respect to the finite extension $k/k^{(d=1)}$. 
Then the following hold:
\begin{enumerate}
  \item[\rm (i)] The automorphism $\alpha^\times$ fits into a commutative diagram of groups
\dd{
k^\times \ar[rr]^-{{\rm Nm}_{k/k^{(d=1)}}} \ar[d]_-{\alpha^\times} && {k^{(d=1)}}^\times \ar@{=}[d] \\
k^\times \ar[rr]_-{{\rm Nm}_{k/k^{(d=1)}}}                         && {k^{(d=1)}}^\times.                
}\label{lemma:2.2:statement:1}
  \item[\rm (ii)] The automorphism $\alpha_{+}$ fits into a commutative diagram of modules
\dd{
k_{+} \ar[rr]^-{{\rm Tr}_{k/k^{(d=1)}}} \ar[d]_-{\alpha_{+}} && k^{(d=1)}_+ \ar@{=}[d] \\
k_{+} \ar[rr]_-{{\rm Tr}_{k/k^{(d=1)}}}                         && k^{(d=1)}_+.                
}
In particular, the automorphism $\alpha_+$ restricts to an  automorphism of \ ${\rm Ker}({\rm Tr}_{k/k^{(d=1)}})$, i.e., the equality $\alpha_+({\rm Ker}({\rm Tr}_{k/k^{(d=1)}})) = {\rm Ker}({\rm Tr}_{k/k^{(d=1)}})$ holds.
\end{enumerate} \label{lemma:2.2:statement:2}
\end{lem}

\begin{proof}
This lemma is none other than
\cite{Hoshi-Nishio}, Lemma 2.3.
\end{proof}

\begin{lem}\label{lemma:2.5}
Suppose that $p_k = 2$, and that $k$ is isomorphic to the finite Galois extension of $k^{(d=1)}$ obtained by adjoining a primitive $4$-th root of unity.  
Write $\nswaut$ for the automorphism of $G_k$ defined in Definition \ref{alpha}, i.e., 
in the case where we take the ``$G$'' of Definition \ref{alpha} to be $G_k$.  
Then, for every nonzero integer $n$, 
the automorphism $\nswaut^n$ 
is not $(\bQ_{2})_{+}$-characteristic. 
\end{lem}
\begin{proof}
It follows from Theorem \ref{theorem:1.5}; \cite{Hoshi1}, Proposition 3.6; \cite{Hoshi2}, Proposition 2.5 that, 
for every nonzero integer $n$, the group-theoretic automorphism 
$\nswaut_+^n$ of $k_+$ is not the identity automorphism but satisfies the equality  $(\nswaut^n_{+} -{\rm id})^2 = 0$ in $\End(k_+)$. 
Here, let us observe that we can write $k = {k^{(d=1)}}(i)$ for some primitive $4$-th root of unity $i$. 
Assume that $\nswaut^n_+$ is $(\bQ_{2})_{+}$-characteristic for some nonzero integer $n$. 
Thus, $\nswaut^n_+(1) = b$ for some $b \in k^{(d=1)}$.  Moreover, it follows from the final portion of Lemma \ref{lemma:2.2}, (ii), that $\nswaut^n_+(i) = ci$ for some $c \in k^{(d=1)}$.  Thus, since $\nswaut_+$ is an automorphism of $\bQ_{2}$-vector space, it follows that, for arbitrary $x$, $y \in k^{(d=1)}$, the equalities 
\begin{eqnarray}
0 = (\nswaut^n_{+} -{\rm id})^2(x+yi) = x (b - 1)^2 + y (c-1)^2 i  \nonumber
\end{eqnarray}
hold. Thus, we have $(b, c) = (1, 1)$. 
In particular, $\nswaut^n_{+}$ is the identity automorphism. However, this is a contradiction. 
\end{proof}

\begin{lem}\label{old:lemma:2.7}
Suppose that $a_k \geq 2$, and that $k$ is absolutely abelian, i.e., that $k$ is Galois over $k^{(d=1)}$, and, moreover, the Galois group ${\rm Gal}(k/k^{(d=1)})$ is abelian [cf.\ \cite{Hoshi2}, Definition 4.2, (ii)]. 
Write $k'$ for the quadratic extension of $k^{(d=1)}$ obtained by adjoining a primitive 4-th root of unity in $k$ and $G_{k'} \defeq {\rm Gal}(\overline{k}/k')$. 
Then the following hold: 
\begin{enumerate}
    \item[\rm (i)] $G_k$ is a characteristic subgroup of $G_{k'}$. In particular, we have a natural homomorphism $\phi \colon {\rm Aut}(G_{k'}) \rightarrow {\rm Aut}(G_k)$. \label{lemma:2.7:statement:1} 
   \item[\rm (ii)] Let $\alpha'$ be an automorphism of $G_{k'}$ that is not $(\bQ_p)_+$-characteristic. Then $\phi(\alpha') \in {\rm Aut}(G_k)$ [cf.\ (i)] is not $(\bQ_p)_+$-characteristic. \label{lemma:2.7:statement:2}
\end{enumerate}
\end{lem}
\begin{proof}
First, we verify assertion (i). 
Let $\beta$ be an automorphism of $G_{k'}$.  Write ${}^* k \subseteq \overline{k}$ 
for the finite extension of $k'$ that corresponds to the open subgroup $\beta(G_{k}) \subseteq G_{k'}$.  Thus, it is immediate 
that $[k: k'] = [{}^*k: k']$.  Moreover, 
since $k$ is absolutely abelian, it follows immediately 
from the main theorem of \cite{JrdRtt} 
[cf.\ also \cite{Hoshi2}, Theorem E, which is 
a generalization of the main theorem of \cite{JrdRtt}] 
that $k$ is contained in ${}^* k$.  
Thus, we conclude from the above equality $[k: k'] = [{}^*k: k']$
that $k = {}^*k$, which thus implies that 
$\beta(G_k) = G_k$, as desired. 
This completes the proof of assertion (i).

Next, we verify assertion (ii). Let us first observe that it follows immediately from the various definitions involved that the diagram 
\dd{
k^{(d=1)}_+ \ar@{^{(}->}[r] \ar[d]_-{\left. \phi(\alpha')_+ \right|_{k^{(d=1)}_+}} & k'_+ \ar@{^{(}->}[r] \ar[d]_-{\alpha'_+} & k_+ \ar[d]^-{\phi(\alpha')_+} \\
\alpha'_+(k^{(d=1)}_+) \ar@{^{(}->}[r]                                                 & k'_+ \ar@{^{(}->}[r]                         & k_+                             
}
commutes, where the horizontal arrows are the natural inclusions, and we write $\alpha'_+$ (resp.\ $\phi(\alpha')_+$) for the group-theoretic automorphism induced by $\alpha' \in \Aut(G_{k'})$ (resp.\ $\phi(\alpha') \in \Aut(G_k)$). 
Since $\alpha'_+$ is not $(\bQ_p)_+$-characteristic, $\alpha'_+(k_+^{(d=1)}) \neq k_+^{(d=1)}$. 
Thus, we conclude from the above diagram that $\phi(\alpha')_+$, hence also $\phi(\alpha')$, is not $(\bQ_p)_+$-characteristic. This completes the proof of assertion (ii). 
\end{proof}

\begin{theorem}\label{one:theorem:2.8}
Let $k$ be an absolutely abelian mixed-characteristic local field such that $p_k = 2$, and $a_k \geq 2$. 
Then there exists an automorphism $\alpha \in {\rm Aut}(G_k)$ such that, 
for each nonzero integer $n$, the automorphism $\alpha^n$ 
is not $(\bQ_{2})_+$-characteristic.  
\end{theorem}
\begin{proof}
Write $k'$ for the quadratic extension of $k^{(d=1)}$ obtained by adjoining a primitive 4-th root of unity in $k$ and $G_{k'} \defeq {\rm Gal}(\overline{k}/k')$. 
Then it follows from Lemma~\ref{lemma:2.5} that
there exists an automorphism $\beta$ of $G_{k'}$ such that, 
for every nonzero integer $n$, the automorphism $\beta^n$ 
is not $(\bQ_{2})_{+}$-characteristic. 
Thus, we conclude from Lemma~\ref{old:lemma:2.7}, (ii), that 
the restriction of $\beta^n$ to $G_k$ is 
not $(\bQ_{2})_+$-characteristic. 
This completes the proof of Theorem~\ref{one:theorem:2.8}. 
\end{proof}

\section{The outer automorphism group of the absolute Galois group of an absolutely abelian 2-adic local field containing primitive $4$-th roots of unity}\label{section3}

In the present \S \ref{section3}, we discuss the outer automorphism
group of the absolute Galois group of a certain 2-adic local field.  In the present
\S \ref{section3}, we maintain the notational conventions introduced
at the beginning of the preceding \S \ref{section2}.  Write, moreover, $\mathrm{Aut}(k)$ for the group of automorphisms of the field $k$ and $\mathrm{Out}(G_k)$ for the group of outer automorphisms of the group $G_k$.  Thus, we have a natural injective [cf.\ \cite{Hoshi1}, Proposition 2.1] homomorphism $\mathrm{Aut}(k) \hookrightarrow \mathrm{Out}(G_k)$ of groups.  In the present \S \ref{section3}, let us regard $\mathrm{Aut}(k)$ as a subgroup of $\mathrm{Out}(G_k)$:  
\dd{
\mathrm{Aut}(k) \subseteq \mathrm{Out}(G_k).  
}

\begin{lem}\label{new:lemma:2.7}
Let $\alpha$ be an automorphism of $G_k$. Suppose that $k$ is absolutely Galois. 
If the image of $\alpha$ in $\mathrm{Out}(G_k)$ is contained in ${\rm N}_{{\rm Out}(G_k)}({\rm Aut}(k))$, then $\alpha$ is $(\bQ_{p_k})_+$-preserving. 
\end{lem}
\begin{proof}
This lemma is none other than 
\cite{Hoshi-Nishio}, Lemma 3.3. 
\end{proof}

\begin{theorem}\label{two:theorem:2.8}
Let $k$ be an absolutely abelian mixed-characteristic local field such that $p_k = 2$, and $a_k \geq 2$. Then the set of ${\rm Out}(G_k)$-conjugates of the subgroup ${\rm Aut}(k) \subseteq {\rm Out}(G_k)$ is infinite. 
\end{theorem}

\begin{proof}
It is immediate that, to verify Theorem~\ref{two:theorem:2.8}, 
it suffices to show that ${\rm N}_{{\rm Out}(G_k)}({\rm Aut}(k))$ 
is of infinite index in ${\rm Out}(G_k)$.  
On the other hand, if $\alpha \in \Aut(G_k)$ is an automorphism 
as in Theorem ~\ref{one:theorem:2.8}, then it follows from 
Lemma \ref{new:lemma:2.7} that,
for each nonzero integer $n$, the image of $\alpha^n$ 
in ${\rm Out}(G_k)$ is not contained in 
${\rm N}_{{\rm Out}(G_k)}({\rm Aut}(k))$.  
In particular, we conclude that 
${\rm N}_{{\rm Out}(G_k)}({\rm Aut}(k))$ 
is of infinite index in ${\rm Out}(G_k)$, as desired.  
This completes the proof of Theorem~\ref{two:theorem:2.8}.  
\end{proof}

\begin{corollary}\label{corollary:2.9}
Let $k$ be an absolutely abelian mixed-characteristic local field such that $p_k = 2$, and $a_k \geq 2$. Then the following hold: 
  \begin{enumerate}
  	\item[\rm (i)] The subgroup ${\rm Aut}(k)$ of ${\rm Out}(G_k)$ is not normal. \label{colloraly:2.9:statement:1}
  	\item[\rm (ii)] There exist infinitely many distinct 
	subgroups of ${\rm Out}(G_k)$ isomorphic to ${\rm Aut}(k)$.  \label{corollary:2.9:statement:2}   
  \end{enumerate}
\end{corollary}
\begin{proof}
These assertions follow immediately from Theorem \ref{two:theorem:2.8}. 
\end{proof}

\begin{corollary}\label{Corollary 3.8} 
Let $k$ be the quadratic mixed-characteristic local field obtained by adjoining a primitive 4-th root of unity to $k^{(d=1)}$ such that $p_k=2$. 
Then the group-theoretic automorphism of $k_+$ induced by an 
automorphism of $G_k$ which lifts an element of the center of 
$\mathrm{Out}(G_k)$ is the identity automorphism of $k_+$.  
\end{corollary}
\begin{proof}
Let $\gamma$ be an element of the center of $\mathrm{Out}(G_k)$.  
Write $\gamma_+ \in \mathrm{Aut}(k_+)$ for the group-theoretic 
automorphism of $k_+$ induced by an automorphism of $G_k$ which 
lifts $\gamma$.  [Note that one verifies easily that $\gamma_+$ 
does not depend on the choice of such a lifting.]  Then it follows from Lemma \ref{new:lemma:2.7} that $\gamma_+$ is $(\mathbb{Q}_{2})_+$-preserving. 
Next, let $\alpha_+ \in \mathrm{Aut}(k_+)$ be a group-theoretic 
automorphism of $k_+$ which is not $(\mathbb{Q}_{2})_+$-characteristic 
[cf.\ Theorem \ref{one:theorem:2.8}]. Then since $\gamma$ is an element 
of the center of $\mathrm{Out}(G_k)$, one verifies immediately that 
$\gamma_+$ commutes with $\alpha_+$.  In particular, since $\gamma_+$ 
is $(\mathbb{Q}_{2})_+$-preserving, $\gamma_+$ restricts to the identity automorphism of $\alpha_+(k_+^{(d=1)}) \subseteq k_+$. 
Thus, since $d_k = 2$, and $k_+^{(d=1)} \neq \alpha_+(k_+^{(d=1)})$, 
we conclude that $\gamma_+$ is the identity automorphism of $k_+$, 
as desired.  This completes the proof of Corollary \ref{Corollary 3.8}. 
\end{proof}

\section{Existence of a special irreducible crystalline representations of dimension two}
In the present \S4, let $k$ be a mixed-characteristic local field and $\bar{k}$ an algebraic closure of $k$. We shall write $G_k \defeq \Gal(\bar{k}/k)$ for the absolute Galois group of $k$ determined by the algebraic closure $\bar{k}$ and ${\rm rec}_k \colon \left(k^{\times}\right)^{\wedge} \xrightarrow{\sim} G^{\rm ab}_k$ for the isomorphism induced by the reciprocity homomorphism $k^{\times} \xrightarrow[]{} \ G^{\rm ab}_k$ in local class field theory. 
Let us recall that, for a given $\bQ_{p_k}$-vector space $V$ of finite dimension and a given continuous representation $\rho: G_k \xrightarrow[]{} \Aut_{\bQ_{p_k}}(V)$ of $G_k$, we say that $\rho$ is \textit{Aut-intrinsically Hodge-Tate} if, for an arbitrary automorphism $\alpha$ of $G_k$, the composite $\rho \circ \alpha :G_k \xrightarrow{\sim} G_k \xrightarrow[]{} \Aut_{\bQ_{p_k}}(V)$
is Hodge-Tate [cf. \cite{Hoshi_intrinsic_Hodge-Tate}, Definition 1.3].

In the present \S4, we prove the existence of an irreducible crystalline [hence also Hodge\textrm{-}Tate] 2-adic representation of dimension 2 that is not \ Aut\textrm{-}intrinsically \ Hodge\textrm{-}Tate [cf. Corollary \ref{cor:4.4} below].

First, let us recall the functorial assignment ``$\mathcal{O}^{\times}(-)$'' of \cite{Hoshi1}, Definition~3.10, (i). Observe that it follows from the functoriality of the assignment ``$\mathcal{O}^{\times}(-)$'' that each automorphism of $G_k$ naturally induces a automorphism of $\mathcal{O}^{\times}(G_k)$. 
In particular, by conjugating this automorphism of $\mathcal{O}^{\times}(G_k)$ by the isomorphism $\mathcal{O}_k^{\times} \xrightarrow{\sim} \mathcal{O}^{\times}(G_k)$ of \cite{Hoshi1}, Proposition~3.11, (i), one concludes that each automorphism of $G_k$ naturally induces a automorphism of $\mathcal{O}_k^{\times}$.
\begin{definition}
    For an automorphism $\alpha$ of $G_k$, we shall write $\alpha_{\times}$ for the automorphism of $\mathcal{O}_k^{\times}$ induced by $\alpha$. 
\end{definition}

\begin{lem}\label{lem:3.1}
Suppose that $k$ is the quadratic mixed-characteristic local field obtained by adjoining a primitive 4-th root of unity to $k^{(d=1)}$, and that $p_k = 2$. 
Write ${\rm Nm} \colon k^\times\;\xrightarrow[]{}\;(k^{(d=1)})^{\times}$
for the norm map with respect to the finite extension $k/k^{(d=1)}$. 
Then the following hold: 
\begin{enumerate}[\rm (i)]
  \item There exists an open submodule $U \subseteq \mathcal{O}^{\times}_k$ such that 
  \begin{enumerate}[\rm (1)]
    \item the topological module $U$ has a natural structure of 
    free $\bZ_2$-module of rank $2$, and, moreover,  
    \item the submodule $U \subseteq \mathcal{O}^{\times}_k$ is preserved by an arbitary automorphism of $\mathcal{O}^{\times}_k$.
  \end{enumerate}
  \item Let $U \subseteq \mathcal{O}^{\times}_k$ be as in (i). Then the topological modules $U \cap \mathcal{O}^{\times}_{k^{(d=1)}}$, $U \cap {\rm Ker}(\Nm)$ have natural structures of free $\bZ_2$-modules of rank $1$, respectively.   
  \item Let $U \subseteq \mathcal{O}^{\times}_k$ be as in (i). Then the equality $U \cap  \mathcal{O}^{\times}_{k^{(d=1)}}  \cap  {\rm Ker}(\Nm)= \{1\}$ holds.  
  \item Let $U \subseteq \mathcal{O}^{\times}_k$ be as in (i). Then the closed submodule of $U$ topologically generated by the closed submodules $U \cap \mathcal{O}^{\times}_{k^{(d=1)}}$ and $U \cap {\rm Ker}(\Nm)$ is open. 
  \item There exists an automorphism $\alpha$ of $G_k$ such that, for an arbitrary nonzero integer $n$, the intersection $\alpha^n_{\times}(\mathcal{O}^{\times}_{k^{(d=1)}}) \cap \mathcal{O}^{\times}_{k^{(d=1)}}$ is not open in $\mathcal{O}^{\times}_{k^{(d=1)}}$. 
  In particular, the automorphism $\alpha^n_{\times}$ does not preserve the submodule $\mathcal{O}^{\times}_{k^{(d=1)}} \subseteq \mathcal{O}^{\times}_k$.  
\end{enumerate}
\end{lem}

\begin{proof}
Assertions (i), (ii) follow from \cite{Hoshi1}, Lemma 1.2, (i) [cf.\ also our assumption that $d_k = p_k =2$.] 
Assertion (iii) is immediate [cf.\ the fact that $U$ is torsion-free, which follows from the condition (1) of assertion (i)]. Assertion (iv) follows from assertions (ii), (iii), together with the condition (1) of assertion (i).

Finally, we verify assertion (v). Let $\alpha$ be an automorphism of $G_k$ as in Definition \ref{alpha}. 
[Note that since $p(G_k) = p_k = 2$, $a(G_k) = a_k = 2 > 1$ --- cf.\ \cite{Hoshi1}, Proposition 3.6; \cite{Hoshi2}, Proposition 2.5, (i) --- 
the group $G_k$ [of MLF-type] satisfies the assumption imposed on ``$G$'' in the discussion preceding Proposition  \ref{proposition:1.1}.] Write $\beta$ for the automorphism of the submodule $U \subseteq \mathscr{O}_k^\times$ obtained by forming the restriction of $\alpha_\times^n$ [cf.\ the condition (2) of assertion (i)]. Thus, since it follows from Theorem \ref{theorem:1.5} [cf.\ also \cite{Hoshi1}, Definition 3.10, (vi)] that
\begin{itemize}
  \item[$(\rm a^\dagger)$] the automorphism of $\mathscr{O}_k^\times \otimes_\mathbb{Z} \mathbb{Q}$ induced by $\alpha_\times^n$ is not the identity automorphism, but
  \item[$(\rm b^\dagger)$] the image of the square of the endomorphism of $\mathscr{O}_k^\times \otimes_\mathbb{Z} \mathbb{Q}$ induced by the endomorphism of $\mathscr{O}_k^\times$ given by ``$a \mapsto \alpha_\times^n(a) \cdot a^{-1}$'' consists of the identity element of $\mathscr{O}_k^\times \otimes_\mathbb{Z} \mathbb{Q}$,
\end{itemize}
one concludes that
\begin{itemize}
  \item[$(\rm a)$] the automorphism $\beta$ is not the identity automorphism of $U$, but
  \item[$(\rm b)$] the image of the square of the endomorphism of $U$ given by ``$a \mapsto \beta(a) \cdot a^{-1}$'' consists of the identity element of $U$.
\end{itemize}
Moreover, it follows immediately from Lemma \ref{lemma:2.2}, (i), that
\begin{itemize}
  \item[$(\rm c)$] the automorphism $\beta$ preserves the submodule $U \cap \mathrm{Ker}(\mathrm{Nm})$ of $U$.
\end{itemize}

Thus, it follows immediately from assertion (iv) [cf.\ also the condition (1) of assertion (i)], together with (b) and (c), that if the automorphism $\beta$ preserves some open submodule of the submodule $U \cap \mathscr{O}_{k^{(d=1)}}^\times$, then $\beta$ is the identity automorphism of $U$ --- in contradiction to (a). In particular, the automorphism $\beta$ does not preserve any open submodule of the submodule $U \cap \mathscr{O}_{k^{(d=1)}}^\times$, which thus implies [cf.\ assertion (ii)] that
\[
\beta(U \cap \mathscr{O}_{k^{(d=1)}}^\times) \cap U \cap \mathscr{O}_{k^{(d=1)}}^\times = \{1\}.
\]
Thus, it follows immediately from \cite{Hoshi1}, Lemma 1.2, (i), that
\[
\alpha_\times^n(\mathscr{O}_{k^{(d=1)}}^\times) \cap \mathscr{O}_{k^{(d=1)}}^\times
\]
is not open in $\mathscr{O}_{k^{(d=1)}}^\times$, as desired. This completes the proof of assertion (v), hence also of Lemma \ref{lem:3.1}. 
\end{proof}

\begin{prop}\label{prop:3.2}
Let $k$ be a mixed-characteristic local field containing a primitive 4-th root of unity such that $p_k=2$. 
Suppose, moreover, that $k$ is absolutely abelian, i.e., that $k$ is absolutely Galois, and the Galois group $\mathrm{Gal}(k/k^{(d=1)})$ is abelian [cf.~ \cite{Hoshi2}, Definition 4.2, (ii)]. Then there exists an automorphism $\alpha$ of $G_k$ such that, for an arbitrary nonzero integer $n$, if one writes $\alpha_\times^n$ for the automorphism of $\mathscr{O}_k^\times$ induced by $\alpha^n$,
then the intersection
\[
\alpha_\times^n(\mathscr{O}_{k^{(d=1)}}^\times) \cap \mathscr{O}_{k^{(d=1)}}^\times
\]
is not open in $\mathscr{O}_{k^{(d=1)}}^\times$. 
\end{prop}

\begin{proof}
Let us first observe that since $k$ contains a primitive 4-th root of unity, it is immediate that $k$ contains the quadratic mixed-characteristic local field obtained by adjoining a primitive 4-th root of unity to $k^{(d=1)}$. 
Moreover, since $k$ is absolutely abelian, it follows immediately from the implication (1) $\Rightarrow$ (2) of \cite{Hoshi2}, Theorem F, (i), that $G_k$ is a characteristic subgroup of the absolute Galois group of the quadratic extension of $k^{(d=1)}$ determined by the algebraic closure $\bar{k}$. Thus, one may conclude that we may assume without loss of generality --- by applying a similar argument to the argument applied in the proof of Lemma \ref{old:lemma:2.7}, (ii), and replacing $k$ by the quadratic extension of $k^{(d=1)}$ --- that $d_k = 2$. On the other hand, if $d_k = 2$, then the desired conclusion follows form Lemma \ref{lem:3.1}, (v). This completes the proof of Proposition \ref{prop:3.2}.
\end{proof}

\begin{theorem} \label{thm:4.3}
Let $k$ be a mixed-characteristic local field containing a primitive 4-th root of unity such that $p_k=2$. 
Assume, moreover, that $k$ is absolutely abelian. 
Then there exist a $\mathbb{Q}_{p_k}$-vector space $V$ of dimension $d_k$ and a continuous representation $\rho : \mathrm{Gal}(\overline{k}/k) \to \mathrm{Aut}_{\mathbb{Q}_{p_k}}(V)$ that is irreducible, abelian, crystalline 
[hence also Hodge-Tate], but not Aut-intrinsically Hodge-Tate [cf.\ \cite{Hoshi_intrinsic_Hodge-Tate}, Definition 1.3].
\end{theorem}

\begin{proof}
Let $\pi \in \mathscr{O}_k$ be a uniformizer of $\mathscr{O}_k$. Write $\rho$ for the continuous representation of $G_k$ [necessarily of dimension $d_k$] obtained by forming the composite
\[
G_k \xtwoheadrightarrow[]{} G_k^{\mathrm{ab}} \xrightarrow{\chi_{\pi,\mathrm{id}_k}} \mathscr{O}_k^\times \xhookrightarrow{\qquad} \mathrm{Aut}_{\mathbb{Q}_{p_k}}(k_+)
\]
— where the first arrow is the natural surjective homomorphism, 
the second arrow is the character defined in \cite{Hoshi_intrinsic_Hodge-Tate}, Definition 1.7, 
and the third arrow is the natural inclusion. Then one verifies easily that this continuous representation $\rho$ is irreducible and abelian. Moreover, it follows immediately from \cite{Serre1998}, Chapter III, §A.4, Proposition 5, that this continuous representation $\rho$ is crystalline.

Next, to verify that the continuous representation $\rho$ is not Aut-intrinsically Hodge-Tate, let us recall that it follows immediately from the various definitions involved that the composite
\[
\mathscr{O}_k^{\times} \xhookrightarrow{\mathrm{rec}_k} G_k^{\mathrm{ab}} \xrightarrow{\chi_{\pi,\mathrm{id}_k}} \mathscr{O}_k^{\times}
\]
is an automorphism that restricts to an automorphism of the submodule $\mathscr{O}_{k^{(d=1)}}^\times \subseteq \mathscr{O}_k^\times$. In particular, if $\alpha$ is an automorphism of $G_k$ as in Proposition \ref{prop:3.2}, then it follows immediately from \cite{Hoshi_intrinsic_Hodge-Tate}, Lemma 1.9, together with the various definitions involved, that the composite $\rho \circ \alpha : G_k \xrightarrow{\sim} G_k \to \mathrm{Aut}_{\mathbb{Q}_{p_k}}(k_+)$ is not Hodge-Tate, which thus implies that the continuous representation $\rho$ is not Aut-intrinsically Hodge-Tate, as desired. This completes the proof of Theorem \ref{thm:4.3}. 
\end{proof}

\begin{corollary} \label{cor:4.4}
There exist a mixed-characteristic local field $K$ such that $p_K = 2$, an algebraic closure $\overline{K}$ of $K$, a $\mathbb{Q}_{p_K}$-vector space $V$ of dimension 2, and a continuous representation $\rho : \mathrm{Gal}(\overline{K}/K) \to \mathrm{Aut}_{\mathbb{Q}_{p_K}}(V)$ that is irreducible, abelian, crystalline [hence also Hodge-Tate], but not Aut-intrinsically Hodge-Tate.
\end{corollary}
\begin{proof}
This assertion is a formal consequence of Theorem \ref{thm:4.3}. 
\end{proof}

\begin{acknowledgements}
The author would like to thank Yuichiro Hoshi for a discussion related to this thesis, and would like to express deepest gratitude to my partner, K. Nishio for providing constant support and warm encouragement.
\end{acknowledgements}

\printbibliography

@article{Diekert1984,
  author  = {V. Diekert},
  title   = {Über die absolute Galoisgruppe dyadischer Zahlkörper},
  journal = {J. Reine Angew. Math.},
  volume  = {350},
  year    = {1984},
  pages   = {152--172}
}

@article{Hoshi-b,
  author       = {Y. Hoshi},
  title        = {A note on the geometricity of open homomorphisms between the absolute Galois groups of $p$-adic local fields},
  journaltitle = {Kodai Math. J.},
  volume       = {36},
  number       = {2},
  year         = {2013},
  pages        = {284--298},
}

@incollection{Hoshi4,
  author = {Y. Hoshi},
  title = {Mono-anabelian reconstruction of number fields},
  booktitle    = {On the examination and further development of inter-universal {T}eichm{\"u}ller theory},
  series       = {RIMS K\^oky\^uroku Bessatsu},
  volume       = {B76},
  year         = {2019},
  location     = {Kyoto},
  publisher    = {Res. Inst. Math. Sci. (RIMS)},
  pages        = {1--77},
}

@article{Hoshi2,
  author = {Y. Hoshi},
  title = {Topics in the anabelian geometry of mixed-characteristic local fields},
  journaltitle = {Hiroshima Math. J.},
  volume       = {49},
  number       = {3},
  year         = {2019},
  pages        = {323--398},
}

@incollection{Hoshi1,
  author = {Y. Hoshi},
  title = {Introduction to mono-anabelian geometry},
 booktitle    = {Publications math{\'e}matiques de Besan{\c c}on. Alg{\`e}bre et th{\'e}orie des nombres. 2021},
  series       = {Publ. Math. Besan{\c c}on Alg{\`e}bre Th{\'e}orie Nr.},
  year         = {2021},
  pages        = {5--44},
  publisher    = {Presses Univ. Franche-Comt{\'e}},
  location     = {Besan{\c c}on},
  addendum     = {[2022]},
}

@article{Hoshi-Nishio,
  author    = {Y. Hoshi and Y. Nishio},
  title     = {On the Outer Automorphism Groups of the Absolute Galois Groups of Mixed-characteristic Local Fields},
  journaltitle = {Res. Number Theory},
  volume       = {8},
  number       = {3},
  year         = {2022},
  month        = apr,
  eid          = {56},
  numpages     = {13},
}

@article{Hoshi_intrinsic_Hodge-Tate,
  title={On intrinsic Hodge-Tate-ness of Galois representations of dimension two},
  author={Y. Hoshi},
  journal={Kodai Math. J.},
  volume={47},
  number={1},
  pages={99--111},
  year={2024},
  publisher={Institute of Science Tokyo, Department of Mathematics}
}

@incollection{JrdRtt,
     author = {W. Jenkner},
     title = {Les corps $p$-adiques dont les groupes de {Galois} absolus sont isomorphes},
     booktitle = {Journ\'ees arithm\'etiques de Gen\`eve - 9-13 septembre 1991},
     editor = {Coray D. F. and P\'etermann Y.-F. S},
     series = {Ast\'erisque},
     pages = {221--226},
     publisher = {Soci\'et\'e math\'ematique de France},
     number = {209},
     year = {1992},
     mrnumber = {1211015},
     zbl = {0804.11065},
     language = {fr},
     url = {https://www.numdam.org/item/AST_1992__209__221_0/}
}

@article{Mzk,
  author = {S. Mochizuki},
  title = {Topics in absolute anabelian geometry I: generalities},
  journal = {J. Math. Sci. Univ. Tokyo},
  volume = {19},
  number = {2},
  pages = {139--242},
  year = {2012}
}

@book{NSW,
  author = {J. Neukirch and A. Schmidt and K. Wingberg},
  title = {Cohomology of number fields},
  edition = {Second},
  series = {Grundlehren der Mathematischen Wissenschaften},
  volume = {323},
  publisher = {Springer-Verlag},
  address = {Berlin},
  year = {2008}
}

@phdthesis{kumpitsch,
  author       = {T. Kumpitsch},
  title        = {Outer automorphisms of the absolute Galois group of local fields of mixed characteristic},
  school       = {Johann Wolfgang Goethe-Universität Frankfurt am Main},
  year         = 2022,
  url          = {https://publikationen.ub.uni-frankfurt.de/opus4/frontdoor/deliver/index/docId/70616/file/PhDThesis_Kumpitsch.pdf},
}

@mastersthesis{kondo,
  author    = {K. Kondo},
  title     = {On the non-normality of the field-theoretic subgroups of the outer automorphism groups of the absolute Galois groups of mixed-characteristic local fields},
  school       = {Kyoto University},
  year         = 2025,
}

@book{Serre1998,
  author       = {J.-P. Serre},
  title        = {Abelian $l$-adic representations and elliptic curves},
  subtitle     = {With the collaboration of Willem Kuyk and John Labute. Revised reprint of the 1968 original},
  series       = {Research Notes in Mathematics},
  volume       = {7},
  publisher    = {A K Peters, Ltd.},
  location     = {Wellesley, MA},
  year         = {1998},
}

\end{document}